\documentclass{article}

\usepackage[margin=1in]{geometry}
\usepackage{amsmath,amssymb,amsthm}
\usepackage{authblk}
\usepackage{tikz}
\usetikzlibrary{decorations.pathreplacing}

\numberwithin{equation}{section}

\usepackage[hidelinks,pagebackref,hypertexnames=false]{hyperref}

\newtheorem{theorem}{Theorem}[section]
\newtheorem{lemma}[theorem]{Lemma}

\begin{document}

\title{Metric entropy of Fourier ratio classes on ${\mathbb Z}_N$}

\author[1]{A.~Iosevich}
\author[2]{V.~Hovhannisyan}
\author[3]{Z.~Keyshams}
\author[4]{A.~Vagharshakyan}

\affil[1]{\small{University of Rochester, NY, USA, \texttt{iosevich@math.rochester.edu}}} 
\affil[2]{\small{Yerevan State University, Yerevan, Armenia, \texttt{vahagn.hv@gmail.com}}}
\affil[3]{\small{Yerevan State University, Yerevan, Armenia, \texttt{zahra.keyshams@ysu.am}}}
\affil[4]{\small{Institute of Mathematics NAS and Yerevan State University, Yerevan, Armenia, \texttt{avaghars@kent.edu}}}

\date{\today}


\date{}

\maketitle
{\def\thefootnote{}
\footnotetext{This work was supported by the Higher Education and Science Committee of RA (Research Project No 24RL-1A028).}
}

\begin{abstract}
We study metric entropy and uniform sampling for classes of signals on ${\mathbb Z}_N$ with prescribed Fourier ratio. The Fourier ratio measures how spread out the Fourier transform of a signal is, interpolating between sparse spectral support and nearly uniform spectral distribution.

Our main result gives upper and lower bounds for the metric entropy of a Fourier-ratio layer of size $r$. At any sufficiently small fixed covering scale, these bounds match in their dependence on $r$ and $N$ and show that $FR(f)^2$ acts as an effective dimension parameter governing the size of the class. We use the entropy estimate to obtain uniform bounds for empirical approximation over Fourier-ratio classes.

We also establish a phase-orbit packing result. If a single signal has a flat spectral block of size $k$, then phase perturbations of that signal generate an exponentially large family with the same Fourier ratio and positive $\ell^2$ separation.

Together, these results show that the Fourier ratio governs not only approximation properties of individual signals, but also the geometric size and uniform sampling behavior of entire signal classes.



\end{abstract}

\section{Introduction}

Let $f:{\mathbb Z}_N \to {\mathbb C}$ be a signal and let
$$
\widehat{f}(m)
=
N^{-\frac{1}{2}}
\sum_{x \in {\mathbb Z}_N}
e^{-2\pi i xm/N} f(x)
$$
denote its discrete Fourier transform. A central problem in harmonic analysis and signal processing is to understand how the distribution of Fourier coefficients controls the complexity of a signal, including its approximation, stability under sampling, and recoverability from incomplete data. The goal of this paper is to identify a simple scalar quantity that captures this complexity and governs the size and statistical behavior of natural classes of signals.

A useful scalar quantity that has recently emerged in this context is the Fourier ratio
$$
FR(f)=\frac{\left|\left|\widehat{f}\right|\right|_{\ell^1}}{\left|\left|\widehat{f}\right|\right|_{\ell^2}},
$$
which takes values in the interval $\left[1,\sqrt{N}\right]$. See, for example, \cite{AldalehEtAl}, \cite{BursteinIosevichNathan}, \cite{IosevichGupta}, and the references contained therein. This quantity interpolates between extreme regimes. When $\widehat{f}$ is supported on a single frequency, $FR(f)=1$. When $\left|\widehat{f}\right|$ is spread uniformly, $FR(f)$ is of order $\sqrt{N}$. In this sense, the Fourier ratio provides a quantitative measure of complexity that captures the transition between spectral sparsity and spectral uniformity.

Recent work shows that the Fourier ratio governs several fundamental approximation-theoretic and structural properties of signals. In particular, \cite{AldalehEtAl} introduced the Fourier ratio and demonstrated that small Fourier ratio implies efficient approximation by low-degree trigonometric polynomials, stability under sampling and perturbation, and additive structure of large spectrum sets. Subsequently, \cite{BursteinIosevichNathan} extended these ideas to algorithmic and learning-theoretic settings, showing that the Fourier ratio controls Kolmogorov rate-distortion complexity and statistical query dimension. These results identify $FR(f)^2$ as a parameter controlling nonlinear approximation, compressibility, and structural behavior.

The purpose of the present paper is to complement these investigations by showing that the Fourier ratio also governs the geometric size and statistical complexity of entire classes of signals. Specifically, we consider the class
$$
C_2(r)
=
\left\lbrace
f:{\mathbb Z}_N \to {\mathbb C} :
\left|\left|f\right|\right|_{\ell^2}=1,
\quad
r \leq \left|\left|\widehat{f}\right|\right|_{\ell^1} \leq 2r
\right\rbrace.
$$

The use of the interval $[r,2r]$ is only a dyadic localization. More general classes with
$$
a \leq \left|\left|\widehat f\right|\right|_{\ell^1} \leq b
$$
can be treated by the same arguments when $a$ and $b$ are comparable, with constants depending only on the comparability ratio $b/a$.

For a metric space $(\mathcal{F}, d)$, let $N_\epsilon(\mathcal{F}, d)$ denote the minimal cardinality of an $\epsilon$-net for $\mathcal{F}$ whose centers belong to $\mathcal{F}$. When the metric is clear from context, we write simply $N_\epsilon(\mathcal{F})$.

We study the metric entropy (i.e., the logarithm of the covering number) and sampling properties of this class. By unitarity of the Fourier transform, this class is equivalent to the class of vectors in ${\mathbb C}^N$ with unit $\ell^2$ norm and $\ell^1$ norm of order $r$. A simple model case is provided by a signal whose Fourier transform is supported on a set of $k$ frequencies with equal magnitude, in which case $FR(f)$ is of order $\sqrt{k}$ and phase perturbations already generate an exponentially large family of well-separated signals with the same Fourier ratio.

Our first main result shows that for $r \in \left[1,\sqrt{N}\right]$ and $\epsilon \in (0,1)$, the metric entropy of this class satisfies
$$
\log N_\epsilon\left(C_2(r)\right)
\le
C\,\frac{r^2}{\epsilon^2}
\left(
\log\left(\frac{eN}{r^2}\right)
+
\log\left(\frac{C}{\epsilon}\right)
\right),
$$
and, for sufficiently small fixed $\epsilon$,
$$
\log N_\epsilon\left(C_2(r)\right)
\ge
c\, r^2
\log\left(\frac{eN}{r^2}\right).
$$

These estimates show that $FR(f)^2$ plays the role of an effective dimension for the class. The upper bound demonstrates that the metric entropy is controlled by $r^2$ up to logarithmic factors, while the lower bound shows that any Fourier-ratio layer contains exponentially many well-separated signals once $r^2$ is large. In particular, the Fourier ratio controls not only the approximation of individual signals, but also the size of the entire class in a quantitative sense.

In addition, we obtain uniform sampling results. If one samples points independently from ${\mathbb Z}_N$, then empirical averages approximate population averages uniformly over $C_2(r)$, provided that the number of samples is sufficiently large in terms of the metric entropy. This connects the Fourier ratio framework to classical ideas from empirical process theory.

Finally, we identify a phase-orbit packing mechanism: if a single signal has a flat spectral block of size $k$, then phase perturbations of that signal generate an exponentially large family of signals with the same Fourier ratio and positive $\ell^2$ separation.

At every sufficiently small fixed covering scale, the upper and lower entropy bounds match in their dependence on $r$ and $N$, up to absolute constants. The remaining gap concerns the dependence on the covering scale $\epsilon$: the upper bound grows like $\epsilon^{-2}$, while the lower bound is presently established only at a fixed scale. Determining the sharp dependence of the metric entropy on $\epsilon$ remains an interesting open problem.

The results of this paper build directly on those in \cite{AldalehEtAl}, where the Fourier ratio is shown to control quantitative trigonometric approximation, stability under sampling and perturbation, and additive structure of large spectrum sets. In that work, the quantity $FR(f)^2$ emerges as a parameter governing approximation and structural properties of individual signals. Where Aldaleh et al. introduced the Fourier ratio and studied its role in approximation and stability of individual signals, the present paper focuses on the metric entropy and sampling complexity of the entire class of signals with a given Fourier ratio.

The present paper shows that the same parameter also governs the size and statistical complexity of the corresponding function classes. In particular, we demonstrate that $FR(f)^2$ controls metric entropy and sampling complexity, thereby extending the role of the Fourier ratio from approximation and structure to statistical learning and recovery from data.

We also point out that the present results fit into a broader program developed in earlier work \cite{BursteinIosevichNathan}, where the Fourier ratio was shown to control algorithmic and learning-theoretic notions of complexity. In particular, it was proved that a bound $FR(f)\le r$ yields explicit upper bounds on Kolmogorov rate–distortion complexity and on the statistical query dimension of the associated function class.

In a companion paper \cite{BursteinIosevichNathan}, the Fourier ratio was introduced as a basis-invariant measure of effective dimension governing signal recovery, localization obstructions, Kolmogorov rate-distortion complexity, and statistical query dimension. The present paper studies a complementary question: the metric entropy and sampling complexity of the Fourier ratio class itself. Together, these papers show that the Fourier ratio simultaneously controls approximation, recovery, structure, statistical complexity, and the geometric size of the associated function classes.

The metric entropy bounds obtained here are consistent with those results and can be viewed as their approximation-theoretic and geometric counterpart. While we do not obtain matching upper and lower bounds in the covering scale $\epsilon$, both estimates identify $FR(f)^2$ as the principal parameter governing the size of Fourier-ratio layers. Indeed, all three notions---metric entropy, algorithmic description length, and statistical query complexity---identify $r^2$ as the fundamental complexity parameter, up to logarithmic factors and the remaining gap in the covering-scale dependence of the metric entropy estimates. This provides further evidence that $FR(f)^2$ is the correct effective dimension governing approximation, recovery, and learning for these classes. A broader basis-invariant formulation of the Fourier-ratio philosophy, including connections with orthonormal systems, localization obstructions, and statistical query complexity, is developed in \cite{BursteinIosevichNathan}.

Although we formulate the results for signals on ${\mathbb Z}_N$, the arguments are fundamentally representation-theoretic and extend with minimal changes to arbitrary finite abelian groups and other finite orthonormal systems.

The distinction is that \cite{BursteinIosevichNathan} controls the complexity of individual objects or associated learning classes through description length and statistical query dimension, while the present paper studies the geometric size of an entire Fourier-ratio layer in $\ell^2$. Thus, the metric entropy estimate below should be viewed as a covering-number analogue of the algorithmic and learning-theoretic bounds obtained there.

A complementary perspective on the role of the Fourier ratio is developed in recent work \cite{IosevichGupta}, where it is shown that small Fourier ratio forces strong additive structure in the set of large values of a signal, in the sense that the indices at which $|f|$ is large can be generated by a small set using $\{-1,0,1\}$ combinations, with the size of the generating set controlled by $FR(f)^2$ up to logarithmic factors.

Taken together, these results support the interpretation of the Fourier ratio as a unifying complexity parameter in discrete harmonic analysis, linking nonlinear approximation, metric entropy, additive structure, and statistical learning. In addition to these class-level results, we show that a single signal whose Fourier transform contains a sufficiently large flat spectral block generates an exponentially large family of signals with identical Fourier ratio and positive $\ell^2$ separation.

\vskip.125in

We now summarize the main results of the paper.

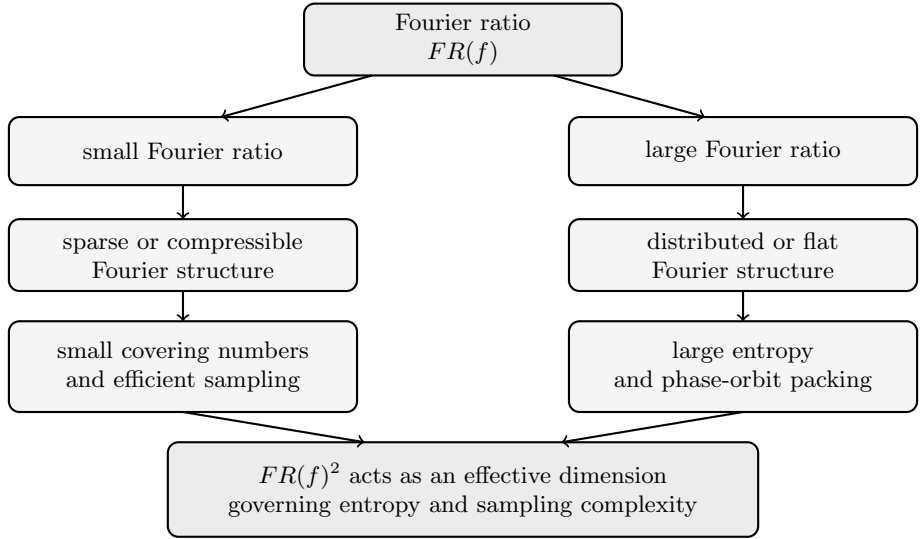
\begin{figure}[ht]
\centering
\begin{tikzpicture}[scale=1.0]

\tikzstyle{every node}=[font=\small]

\draw[thick, rounded corners, fill=gray!15] (3.9,3.0) rectangle (8.1,3.95);
\node[align=center] at (6.0,3.475) {Fourier ratio\\ $FR(f)$};

\draw[thick, rounded corners, fill=gray!8] (0.0,1.55) rectangle (4.6,2.45);
\node[align=center] at (2.3,2.0) {small Fourier ratio};

\draw[thick, rounded corners, fill=gray!8] (0.0,0.15) rectangle (4.6,1.1);
\node[align=center] at (2.3,0.625) {sparse or compressible\\ Fourier structure};

\draw[thick, rounded corners, fill=gray!8] (0.0,-1.45) rectangle (4.6,-0.25);
\node[align=center] at (2.3,-0.85) {small covering numbers\\ and efficient sampling};

\draw[thick, rounded corners, fill=gray!8] (7.4,1.55) rectangle (12.0,2.45);
\node[align=center] at (9.7,2.0) {large Fourier ratio};

\draw[thick, rounded corners, fill=gray!8] (7.4,0.15) rectangle (12.0,1.1);
\node[align=center] at (9.7,0.625) {distributed or flat\\ Fourier structure};

\draw[thick, rounded corners, fill=gray!8] (7.4,-1.45) rectangle (12.0,-0.25);
\node[align=center] at (9.7,-0.85)
{large entropy\\ and phase-orbit packing};

\draw[thick, ->] (4.8,3.0) -- (2.8,2.45);
\draw[thick, ->] (7.2,3.0) -- (9.2,2.45);

\draw[thick, ->] (2.3,1.55) -- (2.3,1.1);
\draw[thick, ->] (2.3,0.15) -- (2.3,-0.25);

\draw[thick, ->] (9.7,1.55) -- (9.7,1.1);
\draw[thick, ->] (9.7,0.15) -- (9.7,-0.25);

\draw[thick, rounded corners, fill=gray!15] (2.1,-3.1) rectangle (9.9,-1.85);

\node[align=center] at (6.0,-2.475)
{$FR(f)^2$ acts as an effective dimension\\
governing entropy and sampling complexity};

\draw[thick, ->] (2.3,-1.45) -- (4.7,-1.85);
\draw[thick, ->] (9.7,-1.45) -- (7.3,-1.85);

\end{tikzpicture}

\caption{The Fourier ratio as an effective dimension parameter. Small Fourier ratio corresponds to compressible spectral structure and leads to small covering numbers and efficient sampling. Large Fourier ratio corresponds to distributed spectral structure and leads to lower-bound mechanisms such as global entropy growth and phase-orbit packing.}

\label{fig:fourier-ratio-roadmap}
\end{figure}

\vskip.125in 

The present paper is part of a broader program investigating Fourier ratio as a quantitative measure of effective spectral complexity across harmonic analysis, approximation theory, sampling, learning, and spectral geometry; see, for example, \cite{AldalehEtAl}, \cite{BursteinIosevichNathan}, \cite{IosevichGupta}, and the references contained therein. The emphasis of the present paper is on metric entropy, sampling complexity, and lower-bound mechanisms in the discrete setting.

The metric entropy bounds for the class $C_2(r)$ are obtained in Theorem \ref{main-entropy-theorem}. The theorem provides an upper bound controlled by $r^2$ up to logarithmic factors and a matching lower bound at fixed scale, showing that Fourier-ratio layers become exponentially large once $r^2$ is large.

The key ingredients are the truncation lemma (Lemma \ref{truncation}), which reduces the problem to sparse vectors, and matching upper and lower entropy bounds for $\Sigma_k$, including the greedy construction (Lemma \ref{greedy-construction}).

Uniform sampling results based on these entropy estimates are established in Theorem \ref{sampling-theorem}, which shows that empirical averages approximate population averages uniformly over $C_2(r)$ with a number of samples controlled by the metric entropy.

Finally, we establish a phase-orbit packing result for individual signals in Lemma \ref{phase-orbit-packing}, which shows that if the Fourier transform of a function contains a flat block of size $k$, then the function generates an exponentially large family of signals with identical Fourier ratio and positive $\ell^2$ separation.

\vskip.125in 
\section{Definitions}

Let $N \geq 1$ and let ${\mathbb Z}_N = {\mathbb Z}/N{\mathbb Z}$. For $f:{\mathbb Z}_N \to {\mathbb C}$ define
$$
\widehat{f}(m)
=
N^{-1/2}
\sum_{x \in {\mathbb Z}_N}
e^{-2\pi i xm/N} f(x).
$$

Then
$$
\left|\left|\widehat{f}\right|\right|_{\ell^2({\mathbb Z}_N)} = \left|\left|f\right|\right|_{\ell^2({\mathbb Z}_N)}.
$$

Define
$$
FR(f) = \frac{\left|\left|\widehat{f}\right|\right|_{\ell^1}}{\left|\left|\widehat{f}\right|\right|_{\ell^2}}.
$$

For $r \geq 1$ define
$$
C_2(r)
=
\left\lbrace
f:{\mathbb Z}_N \to {\mathbb C} :
\left|\left|f\right|\right|_{\ell^2}=1,
\quad
r \leq \left|\left|\widehat{f}\right|\right|_{\ell^1} \leq 2r
\right\rbrace.
$$

By unitarity this is equivalent to
$$
A(r)
=
\left\lbrace
a \in {\mathbb C}^N :
\left|\left|a\right|\right|_2 = 1,
\quad
r \leq \left|\left|a\right|\right|_1 \leq 2r
\right\rbrace.
$$

\section{Auxiliary results}

\begin{lemma}
If $a\in {\mathbb C}^N$ and $\left|\left|a\right|\right|_2 = 1$, then
$$
\left|\left|a\right|\right|_1^2 \leq |\operatorname{supp}(a)|.
$$
\end{lemma}

\begin{proof}
Let $S=\operatorname{supp}(a)$. Then
$$
\left|\left|a\right|\right|_1 = \sum_{j \in S} |a_j|
\le
|S|^{1/2} \left( \sum_{j \in S} |a_j|^2 \right)^{1/2}
=
|S|^{1/2} \left|\left|a\right|\right|_2.
$$
Since $\left|\left|a\right|\right|_2=1$, this gives
$$
\left|\left|a\right|\right|_1 \le |S|^{1/2},
$$
and squaring yields the result.
\end{proof}

\begin{lemma}\label{truncation}
Let $a \in {\mathbb C}^N$ satisfy $\left|\left|a\right|\right|_2 = 1$ and $\left|\left|a\right|\right|_1 \leq 2r$. For every $\epsilon \in (0,1)$ there exists a set $S$ with
$$
|S| \leq \left\lceil \frac{4r^2}{\epsilon^2}\right\rceil
$$
such that
$$
\left|\left|a-a_S\right|\right|_2 \leq \epsilon,
$$
where $a_S$ is the vector obtained by keeping the coordinates in $S$ and setting the rest to zero.
\end{lemma}

\begin{proof}
Let $(a_j^*)_{j=1}^N$ be the non-increasing rearrangement of $(|a_j|)$, so that
$$
a_1^* \geq a_2^* \geq \cdots \geq a_N^* \geq 0.
$$
Set
$$
k
=
\min\left\lbrace
N,
\left\lceil \frac{4r^2}{\epsilon^2}\right\rceil
\right\rbrace,
$$
and let $S$ be the set of indices corresponding to the $k$ largest coordinates of $a$. If $k=N$, then $a_S=a$, so the conclusion is immediate.

Suppose now that $k<N$. Then
$$
\left|\left|a-a_S\right|\right|_2^2
=
\sum_{j>k}(a_j^*)^2.
$$
Since the sequence is decreasing, for $j>k$ we have $a_j^*\leq a_{k+1}^*$, and therefore
$$
\sum_{j>k}(a_j^*)^2
\leq
a_{k+1}^*\sum_{j>k}a_j^*.
$$
Moreover,
$$
a_{k+1}^*
\leq
\frac{1}{k}\sum_{i=1}^k a_i^*
\leq
\frac{\left|\left|a\right|\right|_1}{k},
$$
and
$$
\sum_{j>k}a_j^*\leq \left|\left|a\right|\right|_1.
$$
Hence
$$
\left|\left|a-a_S\right|\right|_2^2
\leq
\frac{\left|\left|a\right|\right|_1^2}{k}
\leq
\frac{4r^2}{k}
\leq
\epsilon^2.
$$
Thus $\left|\left|a-a_S\right|\right|_2\leq\epsilon$. Finally, by the definition of $k$,
$$
|S|=k\leq \left\lceil \frac{4r^2}{\epsilon^2}\right\rceil.
$$
This completes the proof.
\end{proof}

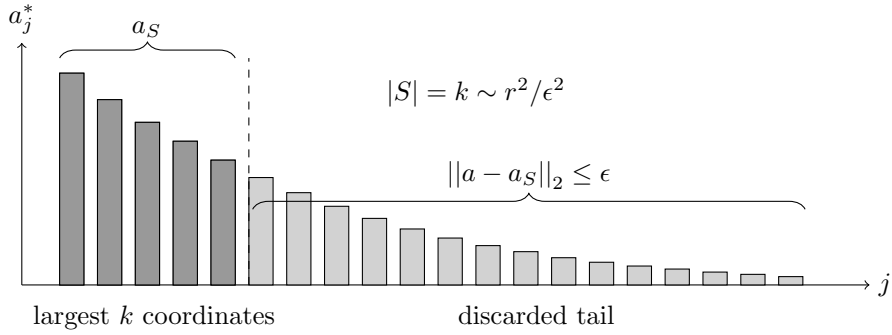
\begin{figure}[ht]
\centering
\begin{tikzpicture}[scale=1.0]

\draw[->] (0,0) -- (11.2,0) node[right] {$j$};
\draw[->] (0,0) -- (0,3.2) node[above] {$a_j^*$};

\foreach \x/\h in {
0.5/2.8,
1.0/2.45,
1.5/2.15,
2.0/1.9,
2.5/1.65,
3.0/1.42,
3.5/1.22,
4.0/1.04,
4.5/0.88,
5.0/0.74,
5.5/0.62,
6.0/0.52,
6.5/0.44,
7.0/0.36,
7.5/0.30,
8.0/0.25,
8.5/0.21,
9.0/0.17,
9.5/0.14,
10.0/0.11}
{
\draw[fill=gray!35] (\x,0) rectangle +(0.32,\h);
}

\foreach \x/\h in {
0.5/2.8,
1.0/2.45,
1.5/2.15,
2.0/1.9,
2.5/1.65}
{
\draw[fill=gray!80] (\x,0) rectangle +(0.32,\h);
}

\draw[dashed] (3.0,0) -- (3.0,3.0);
\node[below] at (1.75,-0.15) {largest $k$ coordinates};
\node[below] at (6.8,-0.15) {discarded tail};

\draw[decorate, decoration={brace, amplitude=6pt}] (0.5,3.0) -- (2.82,3.0);
\node[above] at (1.65,3.15) {$a_S$};

\draw[decorate, decoration={brace, amplitude=6pt}] (3.05,1.0) -- (10.35,1.0);
\node[above] at (6.7,1.15) {$\left|\left|a-a_S\right|\right|_2 \le \epsilon$};

\node at (6.0,2.55) {$|S|=k \sim r^2/\epsilon^2$};

\end{tikzpicture}
\caption{Truncation to the largest $k$ Fourier coefficients. A vector with $\ell^2$ norm $1$ and $\ell^1$ norm at most $2r$ is approximated by retaining its largest $k \sim r^2/\epsilon^2$ coordinates. The discarded tail has $\ell^2$ norm at most $\epsilon$.}
\end{figure}

\begin{lemma}
Let
$$
\Sigma_k = \left\lbrace a \in {\mathbb C}^N : \left|\left|a\right|\right|_2 = 1, \ |\operatorname{supp}(a)| \leq k\right\rbrace.
$$
Then
$$
\log N_\epsilon(\Sigma_k)
\leq
k \log\left(\frac{eN}{k}\right)
+
2k \log\left(\frac{C}{\epsilon}\right).
$$
\end{lemma}

\begin{proof}
We give a brief self-contained covering argument; see also \cite[Section 4.2]{Vershynin}.
For each subset $S \subset \{1,\cdots,N\}$ with $|S| = k$, define
\[
\Sigma_k(S)
=
\left\lbrace
a \in \mathbb{C}^N : \left|\left|a\right|\right|_2 = 1,\ \operatorname{supp}(a) \subset S
\right\rbrace
\]
Then
\[
\Sigma_k \subset \bigcup_{|S|=k} \Sigma_k(S).
\]
For each fixed $S$, the set $\Sigma_k(S)$ is isometric to the unit sphere in a $2k$-dimensional Euclidean space. By standard volumetric estimates (see \cite[Corollary 4.2.13]{Vershynin}), there exists an $\varepsilon$-net $\mathcal N_S$ satisfying
\[
|\mathcal N_S| \le \left( \frac{C}{\varepsilon} \right)^{2k}.
\]
Define
\[
\mathcal N := \bigcup_{|S|=k} \mathcal N_S.
\]
The number of subsets $S$ satisfies
\[
\binom{N}{k} \le \left( \frac{eN}{k} \right)^k.
\]
Therefore,
\[
|\mathcal N|
\le
\binom{N}{k} \left( \frac{C}{\varepsilon} \right)^{2k}
\le \left( \frac{eN}{k} \right)^k \left( \frac{C}{\varepsilon} \right)^{2k}.
\]
Taking logarithms gives
\[
\log N_\varepsilon(\Sigma_k)
\le
k \log\!\left( \frac{eN}{k} \right)
+
2k \log\!\left( \frac{C}{\varepsilon} \right).
\]
\end{proof}

\begin{lemma}\label{greedy-construction}
There exists an absolute constant $c>0$ such that, for every $1\leq k\leq N$, there is a set ${\mathcal P}_k\subset\Sigma_k$ with the following properties. Every vector in ${\mathcal P}_k$ has exactly $k$ nonzero coordinates, all of magnitude $k^{-1/2}$; distinct vectors in ${\mathcal P}_k$ have distance at least $1$ in $\ell^2$; and
$$
\log |{\mathcal P}_k|
\geq
c k\log\left(\frac{eN}{k}\right).
$$
Consequently, for every $0<\epsilon<1/2$,
$$
\log N_\epsilon(\Sigma_k)
\geq
c k\log\left(\frac{eN}{k}\right).
$$
\end{lemma}

\begin{proof}
Fix a sufficiently large absolute constant $A>1$. We first consider the case $k\leq N/A$.

If $k\leq3$, choose a family of pairwise disjoint $k$-element subsets of $\{1,\dots,N\}$. The normalized indicator vectors of these sets are pairwise separated by distance $\sqrt{2}$, and there are at least $\lfloor N/k\rfloor$ of them. After decreasing the absolute constant $c$, this gives the required estimate.

Suppose now that $4\leq k\leq N/A$. We construct a family ${\mathcal F}$ of $k$-element subsets of $\{1,\dots,N\}$ such that
$$
|S\cap T|\leq \frac{k}{2}
$$
whenever $S,T\in{\mathcal F}$ are distinct. Start with all $k$-element subsets and choose them greedily. After selecting a set $S$, discard every set $T$ satisfying $|S\cap T|>k/2$. If $j=k-|S\cap T|$, then the number of discarded sets is at most
$$
1+
\sum_{1\leq j<k/2}
\binom{k}{j}\binom{N-k}{j}.
$$
Using $\binom{k}{j}\leq2^k$ and
$$
\binom{N-k}{j}
\leq
\binom{N}{j}
\leq
\left(\frac{eN}{j}\right)^j,
$$
we obtain
$$
1+
\sum_{1\leq j<k/2}
\binom{k}{j}\binom{N-k}{j}
\leq
k2^k\left(\frac{2eN}{k}\right)^{k/2}.
$$
Since
$$
\binom{N}{k}
\geq
\left(\frac{N}{k}\right)^k,
$$
the greedy construction produces a family satisfying
$$
|{\mathcal F}|
\geq
\frac{\binom{N}{k}}
{k2^k\left(\frac{2eN}{k}\right)^{k/2}}
\geq
\exp\left(
\frac{k}{2}\log\left(\frac{N}{k}\right)-Ck
\right)
$$
for an absolute constant $C>0$. Choosing $A$ sufficiently large gives
$$
|{\mathcal F}|
\geq
\exp\left(
c k\log\left(\frac{eN}{k}\right)
\right).
$$
For each $S\in{\mathcal F}$, define $a^{(S)}\in{\mathbb C}^N$ by
$$
a^{(S)}_j
=
\begin{cases}
k^{-1/2}, & j\in S,\\
0, & j\notin S.
\end{cases}
$$
Then $a^{(S)}\in\Sigma_k$, and for distinct $S,T\in{\mathcal F}$,
$$
\left|\left|a^{(S)}-a^{(T)}\right|\right|_2^2
=
2-2\frac{|S\cap T|}{k}
\geq
1.
$$

It remains to consider $N/A<k\leq N$. Fix a set $S\subset\{1,\dots,N\}$ with $|S|=k$. A standard greedy packing of the Hamming cube produces a family ${\mathcal G}\subset\{-1,1\}^S$ such that
$$
|{\mathcal G}|\geq\exp(c_0k)
$$
and any two distinct elements of ${\mathcal G}$ differ in at least $k/4$ coordinates. For $\sigma\in{\mathcal G}$, define
$$
a^{(\sigma)}_j
=
\begin{cases}
\sigma(j)k^{-1/2}, & j\in S,\\
0, & j\notin S.
\end{cases}
$$
Then $a^{(\sigma)}\in\Sigma_k$, and for distinct $\sigma,\tau\in{\mathcal G}$,
$$
\left|\left|a^{(\sigma)}-a^{(\tau)}\right|\right|_2^2
\geq
\frac{k}{4}\cdot\frac{4}{k}
=
1.
$$
Since $k>N/A$, the quantity $\log(eN/k)$ is bounded by the absolute constant $\log(eA)$. Thus, after decreasing $c$ if necessary,
$$
|{\mathcal G}|
\geq
\exp\left(
c k\log\left(\frac{eN}{k}\right)
\right).
$$
The asserted covering-number estimate follows because an $\epsilon$-ball with $0<\epsilon<1/2$ contains at most one point of a $1$-separated set.
\end{proof}

\vskip.125in
We now complement the global entropy and sampling results above with a phase-orbit construction that produces large, well-separated families of signals starting from a single function. The results above show that the class $C_2(r)$ is large in a metric and statistical sense. However, these arguments do not directly address whether a given function $f \in C_2(r)$ can itself generate a geometrically rich family of signals with the same Fourier ratio.

The key mechanism is the presence of a sufficiently large collection of Fourier coefficients of comparable magnitude. In this situation, one can modify the phases of the Fourier transform on this set without changing either the $\ell^2$ norm or the $\ell^1$ norm. This produces a family of signals with the same Fourier ratio as $f$, while maintaining quantitative separation in $\ell^2$.

The lemma below formalizes this idea by showing that a flat spectral block gives rise to an exponentially large family of signals obtained by phase perturbations, all of which remain in the same Fourier ratio layer and are separated by a fixed positive distance. This provides a phase-orbit packing generated by a single signal.

\begin{lemma}[Phase-orbit packing]\label{phase-orbit-packing}
Let $f:{\mathbb Z}_N\to{\mathbb C}$ satisfy $\left|\left|f\right|\right|_{\ell^2}=1$. Suppose that there exists a set
$S\subset{\mathbb Z}_N$ with $|S|=k$ and a constant $a>0$ such that
$$
\left|\widehat f(m)\right|\ge a k^{-1/2}
\quad \text{for all } m\in S.
$$
For each choice of signs $\sigma:S\to\{-1,1\}$ define $f_\sigma$ by
$$
\widehat f_\sigma(m)=
\begin{cases}
\sigma(m)\widehat f(m) & m\in S,\\
\widehat f(m) & m\notin S.
\end{cases}
$$
Then $\left|\left|f_\sigma\right|\right|_{\ell^2}=1$ and
$$
\left|\left|\widehat f_\sigma\right|\right|_{\ell^1}=\left|\left|\widehat f\right|\right|_{\ell^1}.
$$
Moreover, there exists a subfamily $\Sigma\subset\{-1,1\}^S$ satisfying
$$
|\Sigma|\ge \exp(c k)
$$
such that for distinct $\sigma,\tau\in\Sigma$,
$$
\left|\left|f_\sigma-f_\tau\right|\right|_{\ell^2}\ge a,
$$
where $c>0$ is an absolute constant.
\end{lemma}

\begin{proof}
Since the Fourier transform is unitary, it is enough to estimate the distance between the Fourier transforms. For every $\sigma$,
$$
\left|\left|\widehat f_\sigma\right|\right|_{\ell^2}=\left|\left|\widehat f\right|\right|_{\ell^2}=1,
$$
because changing signs does not change the absolute values of the Fourier coefficients. Similarly,
$$
\left|\left|\widehat f_\sigma\right|\right|_{\ell^1}=\left|\left|\widehat f\right|\right|_{\ell^1}.
$$

We now construct a large separated family of signs. Identify $\{-1,1\}^S$ with the vertices of the Hamming cube of dimension $k$. For all sufficiently large $k$, choose signs greedily as follows. Select one vertex, and discard all vertices at Hamming distance less than $k/4$ from it. Repeat this procedure until no vertices remain.

The number of vertices discarded at each step is at most
$$
\sum_{j<k/4} \binom{k}{j}.
$$
For $k$ sufficiently large this quantity is at most $\exp(c_1 k)$ with $c_1<\log 2$. Since the whole cube has cardinality $2^k$, the greedy procedure produces a set $\Sigma\subset\{-1,1\}^S$ satisfying
$$
|\Sigma|\ge \exp(c k)
$$
such that any two distinct $\sigma,\tau\in\Sigma$ differ on at least $k/4$ elements of $S$.

For the finitely many remaining values of $k$, take the two constant sign patterns $\sigma\equiv 1$ and $\tau\equiv -1$. After decreasing the absolute constant $c>0$, the estimate $|\Sigma|\geq\exp(ck)$ holds for every $k\geq1$.

For distinct $\sigma$ and $\tau$,
$$
\left|\left|f_\sigma-f_\tau\right|\right|_{\ell^2}^2
=
\left|\left|\widehat f_\sigma-\widehat f_\tau\right|\right|_{\ell^2}^2.
$$
The two Fourier transforms differ only on $S$. On every point $m\in S$ where $\sigma(m)\ne\tau(m)$, one has
$$
|\widehat f_\sigma(m)-\widehat f_\tau(m)|
=
2|\widehat f(m)|.
$$
In the large-$k$ case, the sign patterns differ on at least $k/4$ elements of $S$, and therefore
$$
\left|\left|f_\sigma-f_\tau\right|\right|_{\ell^2}^2
\ge
\frac{k}{4}\cdot 4a^2 k^{-1}
=
a^2.
$$
In the remaining finite cases, the two constant sign patterns differ on every point of $S$, and hence
$$
\left|\left|f_\sigma-f_\tau\right|\right|_{\ell^2}^2
\geq
k\cdot 4a^2 k^{-1}
=
4a^2.
$$
Thus in all cases
$$
\left|\left|f_\sigma-f_\tau\right|\right|_{\ell^2}\ge a.
$$
This completes the proof.
\end{proof}

\vskip.125in

\section{Main entropy theorem}

\begin{theorem}\label{main-entropy-theorem}
Let $1 \leq r \leq \sqrt{N}$ and $0<\epsilon<1$. Then
$$
\log N_\epsilon(C_2(r))
\leq
C\frac{r^2}{\epsilon^2}
\left(
\log\left(\frac{eN}{r^2}\right)
+
\log\left(\frac{C}{\epsilon}\right)
\right),
$$
where $C>0$ is an absolute constant.

Moreover, there exists an absolute constant $\epsilon_0>0$ such that, for every $0<\epsilon<\epsilon_0$,
$$
\log N_\epsilon(C_2(r))
\geq
c r^2\log\left(\frac{eN}{r^2}\right),
$$
where $c>0$ is an absolute constant.
\end{theorem}

\begin{proof}
By unitarity of the Fourier transform, it suffices to prove the corresponding statement for
$$
A(r)
=
\left\lbrace
a\in{\mathbb C}^N:
\left|\left|a\right|\right|_2=1,
\quad
r\leq \left|\left|a\right|\right|_1\leq 2r
\right\rbrace.
$$

We first prove the upper bound. Apply Lemma \ref{truncation} with $\epsilon/8$ in place of $\epsilon$. Then every $a\in A(r)$ admits a truncation $a_S$ supported on at most
$$
k
=
\min\left\lbrace
N,
\left\lceil\frac{256r^2}{\epsilon^2}\right\rceil
\right\rbrace
$$
coordinates such that
$$
\left|\left|a-a_S\right|\right|_2\leq \frac{\epsilon}{8}.
$$
Set
$$
b=\frac{a_S}{\left|\left|a_S\right|\right|_2}.
$$
Then $b\in\Sigma_k$. Since $a_S$ and $a-a_S$ have disjoint supports,
$$
1-\left|\left|a_S\right|\right|_2^2
=
\left|\left|a-a_S\right|\right|_2^2.
$$
Consequently,
$$
1-\left|\left|a_S\right|\right|_2
\leq
\left|\left|a-a_S\right|\right|_2^2
\leq
\frac{\epsilon^2}{64}.
$$
Therefore,
$$
\left|\left|a-b\right|\right|_2
\leq
\left|\left|a-a_S\right|\right|_2+\left|\left|a_S-b\right|\right|_2
\leq
\frac{\epsilon}{8}+\frac{\epsilon^2}{64}
\leq
\frac{\epsilon}{4}.
$$

Let ${\mathcal N}$ be an $\epsilon/4$-net for $\Sigma_k$ whose centers belong to $\Sigma_k$ and whose cardinality is $N_{\epsilon/4}(\Sigma_k)$. For every $z\in{\mathcal N}$ such that
$$
B\left(z,\frac{\epsilon}{2}\right)\cap A(r)\neq\varnothing,
$$
choose one point $a_z\in A(r)$ satisfying
$$
\left|\left|a_z-z\right|\right|_2\leq\frac{\epsilon}{2}.
$$
We claim that the collection of all such points $a_z$ is an $\epsilon$-net for $A(r)$. Indeed, given $a\in A(r)$, choose $b\in\Sigma_k$ as above and then choose $z\in{\mathcal N}$ with
$$
\left|\left|b-z\right|\right|_2\leq\frac{\epsilon}{4}.
$$
Then
$$
\left|\left|a-z\right|\right|_2
\leq
\left|\left|a-b\right|\right|_2+\left|\left|b-z\right|\right|_2
\leq
\frac{\epsilon}{2},
$$
so the point $a_z$ is defined, and
$$
\left|\left|a-a_z\right|\right|_2
\leq
\left|\left|a-z\right|\right|_2+\left|\left|z-a_z\right|\right|_2
\leq
\epsilon.
$$
Thus
$$
N_\epsilon(A(r))
\leq
N_{\epsilon/4}(\Sigma_k).
$$

If $k=N$, then $\Sigma_k$ is the unit sphere in ${\mathbb C}^N$. The standard volumetric estimate gives
$$
\log N_{\epsilon/4}(\Sigma_k)
\leq
2N\log\left(\frac{C}{\epsilon}\right).
$$
Since $k=N$, we have
$$
N
\leq
\left\lceil\frac{256r^2}{\epsilon^2}\right\rceil
\leq
257\frac{r^2}{\epsilon^2}.
$$
Hence
$$
\log N_\epsilon(A(r))
\leq
C\frac{r^2}{\epsilon^2}
\log\left(\frac{C}{\epsilon}\right),
$$
which is bounded by the asserted upper estimate.

We may therefore assume that $k<N$. Then
$$
k
=
\left\lceil\frac{256r^2}{\epsilon^2}\right\rceil
\leq
257\frac{r^2}{\epsilon^2}.
$$
The covering estimate for $\Sigma_k$ gives
$$
\log N_{\epsilon/4}(\Sigma_k)
\leq
k\log\left(\frac{eN}{k}\right)
+
2k\log\left(\frac{C}{\epsilon}\right).
$$
Since $k\geq r^2$, we have
$$
\log\left(\frac{eN}{k}\right)
\leq
\log\left(\frac{eN}{r^2}\right).
$$
Combining these estimates yields
$$
\log N_\epsilon(A(r))
\leq
C\frac{r^2}{\epsilon^2}
\left(
\log\left(\frac{eN}{r^2}\right)
+
\log\left(\frac{C}{\epsilon}\right)
\right).
$$

We now prove the lower bound. Set
$$
k=\left\lceil r^2\right\rceil.
$$
Then $1\leq k\leq N$ and
$$
r
\leq
\sqrt{k}
\leq
\sqrt{r^2+1}
\leq
\sqrt{2}\,r
\leq
2r.
$$
By Lemma \ref{greedy-construction}, there exists a set ${\mathcal P}_k\subset\Sigma_k$ such that every vector in ${\mathcal P}_k$ has exactly $k$ nonzero coordinates, all of magnitude $k^{-1/2}$, distinct vectors in ${\mathcal P}_k$ have distance at least $1$, and
$$
\log \left|{\mathcal P}_k\right|
\geq
c k\log\left(\frac{eN}{k}\right).
$$
Every vector in ${\mathcal P}_k$ has $\ell^1$ norm equal to $\sqrt{k}$, and therefore ${\mathcal P}_k\subset A(r)$. Hence, for every $0<\epsilon<1/2$,
$$
\log N_\epsilon(A(r))
\geq
\log |{\mathcal P}_k|
\geq
c k\log\left(\frac{eN}{k}\right).
$$
The function
$$
t\longmapsto t\log\left(\frac{eN}{t}\right)
$$
is nondecreasing on $[1,N]$. Since $k\geq r^2$, it follows that
$$
k\log\left(\frac{eN}{k}\right)
\geq
r^2\log\left(\frac{eN}{r^2}\right).
$$
Thus
$$
\log N_\epsilon(A(r))
\geq
c r^2\log\left(\frac{eN}{r^2}\right)
$$
for every $0<\epsilon<1/2$. This proves the lower bound with $\epsilon_0=1/2$ and completes the proof.
\end{proof}

\section{Sampling complexity via metric entropy}

Let $\mu$ denote the uniform probability measure on ${\mathbb Z}_N$, that is,
$$
\mu(E)=\frac{|E|}{N}.
$$

For $\omega=(x_1,\dots,x_m)\in({\mathbb Z}_N)^m$, define
$$
\left|\left|f\right|\right|_{L^2(\mu_\omega)}^2
=
\frac{1}{m}\sum_{j=1}^m |f(x_j)|^2.
$$

\begin{theorem}\label{sampling-theorem}
Let $0<\delta<1$ and $0<\rho<1$. There exists an absolute constant $C>0$ such that if
$$
m
\geq
C\frac{r^4}{\delta^2}
\left(
\log N_{\delta/(16r\sqrt{N})}(C_2(r),\left|\left|\cdot\right|\right|_2)
+
\log\frac{1}{\rho}
\right),
$$
then, with probability at least $1-\rho$ with respect to $\mu^{\otimes m}$,
$$
\sup_{f\in C_2(r)}
\left|
\frac{N}{m}\sum_{j=1}^m |f(x_j)|^2
-
\sum_{x\in{\mathbb Z}_N}|f(x)|^2
\right|
\leq
\delta.
$$
Since every $f\in C_2(r)$ satisfies $\left|\left|f\right|\right|_{\ell^2}=1$, the second term in the preceding display equals $1$.
\end{theorem}

\begin{proof}
If $f\in C_2(r)$, then $\left|\left|\widehat f\right|\right|_2=1$ and $\left|\left|\widehat f\right|\right|_1\leq 2r$. By Fourier inversion,
$$
f(x)
=
N^{-1/2}
\sum_{m\in{\mathbb Z}_N}
e^{2\pi i xm/N}\widehat f(m),
$$
and therefore
$$
\left|\left|f\right|\right|_{L^\infty}
\leq
N^{-1/2}\left|\left|\widehat f\right|\right|_{\ell^1}
\leq
\frac{2r}{\sqrt N}.
$$
Hence
$$
0\leq N|f(x)|^2\leq 4r^2
$$
for every $x\in{\mathbb Z}_N$.

Fix $g\in C_2(r)$. Since
$$
\mathbb E\bigl[N|g(x_j)|^2\bigr]
=
\sum_{x\in{\mathbb Z}_N}|g(x)|^2
=
1,
$$
Hoeffding's inequality gives
$$
\mathbb P
\left(
\left|
\frac{N}{m}\sum_{j=1}^m |g(x_j)|^2
-
\sum_{x\in{\mathbb Z}_N}|g(x)|^2
\right|
>
\frac{\delta}{2}
\right)
\leq
2\exp\left(-c\frac{m\delta^2}{r^4}\right).
$$

Let ${\mathcal N}$ be an $\alpha$-net for $C_2(r)$ in $\ell^2$, where
$$
\alpha=\frac{\delta}{16r\sqrt N}.
$$
Applying the preceding estimate to each $g\in{\mathcal N}$ and using the union bound, we obtain that, outside an exceptional set of probability at most
$$
2|{\mathcal N}|\exp\left(-c\frac{m\delta^2}{r^4}\right),
$$
one has
$$
\left|
\frac{N}{m}\sum_{j=1}^m |g(x_j)|^2
-
\sum_{x\in{\mathbb Z}_N}|g(x)|^2
\right|
\leq
\frac{\delta}{2}
$$
for every $g\in{\mathcal N}$.

Now fix $f\in C_2(r)$ and choose $g\in{\mathcal N}$ such that
$$
\left|\left|f-g\right|\right|_{\ell^2}\leq\alpha.
$$
Since the Fourier transform is unitary,
$$
\left|\left|\widehat f-\widehat g\right|\right|_{\ell^2}
=
\left|\left|f-g\right|\right|_{\ell^2}
\leq
\alpha.
$$
By Fourier inversion and Cauchy--Schwarz,
$$
\left|\left|f-g\right|\right|_{L^\infty}
\leq
\left|\left|\widehat f-\widehat g\right|\right|_{\ell^2}
\leq
\alpha.
$$
Also,
$$
\left|\left|f\right|\right|_{L^\infty},\left|\left|g\right|\right|_{L^\infty}
\leq
\frac{2r}{\sqrt N}.
$$
Therefore, for every $x\in{\mathbb Z}_N$,
$$
N\bigl||f(x)|^2-|g(x)|^2\bigr|
\leq
N\bigl(|f(x)|+|g(x)|\bigr)|f(x)-g(x)|
\leq
4r\sqrt N\,\alpha
=
\frac{\delta}{4}.
$$
It follows that
$$
\left|
\frac{N}{m}\sum_{j=1}^m
\bigl(|f(x_j)|^2-|g(x_j)|^2\bigr)
\right|
\leq
\frac{\delta}{4}.
$$
Moreover, since $\left|\left|f\right|\right|_{\ell^2}=\left|\left|g\right|\right|_{\ell^2}=1$,
$$
\sum_{x\in{\mathbb Z}_N}
\bigl(|f(x)|^2-|g(x)|^2\bigr)
=
0.
$$
Combining these estimates with the estimate for $g\in{\mathcal N}$ gives
$$
\left|
\frac{N}{m}\sum_{j=1}^m |f(x_j)|^2
-
\sum_{x\in{\mathbb Z}_N}|f(x)|^2
\right|
\leq
\frac{3\delta}{4}
\leq
\delta.
$$
Since $f\in C_2(r)$ was arbitrary, the desired uniform estimate follows.

Finally, choosing $m$ so that
$$
2|{\mathcal N}|
\exp\left(-c\frac{m\delta^2}{r^4}\right)
\leq
\rho
$$
ensures that the exceptional set has probability at most $\rho$. This condition is implied by the stated lower bound on $m$, after adjusting the absolute constant $C$. This completes the proof.
\end{proof}

\section{Conclusion and future work}

The results of this paper demonstrate that the Fourier ratio $FR(f)$ governs not only the approximation-theoretic properties of individual signals but also the metric entropy and uniform sampling behavior of the associated function classes. The quantity $FR(f)^2$ emerges as an effective dimension parameter controlling the geometric size of Fourier-ratio layers.

Several natural questions remain open for future investigation.

\begin{enumerate}
\item \textbf{Dependence on the covering scale.} At every sufficiently small fixed covering scale, the upper and lower entropy bounds match in their dependence on $r$ and $N$, up to absolute constants. It remains open to determine the sharp dependence of $\log N_\epsilon(C_2(r))$ on $\epsilon$ as $\epsilon$ tends to zero.

\item \textbf{Continuous and infinite-dimensional settings.} Our arguments extend directly to finite abelian groups, but what about functions on $[0,1]$ with bandlimited Fourier transform? More generally, for orthonormal systems $\{\phi_j\}_{j=1}^\infty$ with $\left|\left|\phi_j\right|\right|_\infty \le B$, the truncation argument still works, but the covering estimates become more delicate due to infinite dimensionality.

\item \textbf{Phase-orbit packing beyond flat blocks.} The phase-orbit packing result (Lemma \ref{phase-orbit-packing}) assumes a flat spectral block. Can one obtain similar lower bounds under weaker conditions, such as a lower bound on the $\ell^4$ norm of the Fourier coefficients?

\item \textbf{Algorithmic implications.} The metric entropy bounds suggest that empirical risk minimization over $C_2(r)$ should have favorable sample complexity. Developing efficient algorithms that achieve these bounds is an important practical direction.
\end{enumerate}


\vskip.25in 

\begin{thebibliography}{99}

\bibitem{AldalehEtAl}
K.~Aldaleh, W.~Burstein, G.~Garza, G.~Hart, A.~Iosevich, J.~Iosevich,
A.~Khalil, J.~King, N.~Kulkarni, T.~Le, I.~Li, A.~Mayeli, B.~McDonald,
K.~Nguyen, and N.~Shaffer,
\textit{The Fourier Ratio and Quantitative Trigonometric Approximation},
arXiv:2511.19560, 2026.

\bibitem{AnthonyBartlett}
M.~Anthony and P.~Bartlett,
\textit{Neural Network Learning: Theoretical Foundations},
Cambridge University Press, Cambridge, 1999.

\bibitem{BoucheronLugosiMassart}
S.~Boucheron, G.~Lugosi, and P.~Massart,
\textit{Concentration Inequalities},
Oxford University Press, Oxford, 2013.

\bibitem{BursteinIosevichNathan}
W.~Burstein, A.~Iosevich, and H.~S.~Nathan,
\textit{The Fourier Ratio: A Unifying Measure of Complexity for Recovery, Localization, and Learning},
arXiv:2601.16345, 2026.

\bibitem{FoucartRauhut}
S.~Foucart and H.~Rauhut,
\textit{A Mathematical Introduction to Compressive Sensing},
Birkh\"auser, New York, 2013.

\bibitem{Goerg2013}
G.~M.~Goerg,
\textit{Forecastable Component Analysis},
Proceedings of the 30th International Conference on Machine Learning (ICML),
2013, pp.~64--72.

\bibitem{IosevichGupta}
A.~Iosevich and V.~Gupta,
\textit{Large values in time series and additive combinatorics},
arXiv:2604.21292, 2026.

\bibitem{Vershynin}
R.~Vershynin,
\textit{High-Dimensional Probability: An Introduction with Applications in Data Science},
Cambridge University Press, Cambridge, 2018.

\end{thebibliography}
\end{document}